\documentclass[pdflatex,sn-mathphys-num]{sn-jnl}% Math and Physical Sciences Numbered Reference Style
\usepackage{graphicx}%
\usepackage{multirow}%
\usepackage{amsmath,amssymb,amsfonts}%
\usepackage{amsthm}%
\usepackage{mathrsfs}%
\usepackage[title]{appendix}%
\usepackage{xcolor}%
\usepackage{textcomp}%
\usepackage{manyfoot}%
\usepackage{booktabs}%
\usepackage{algorithm}%
\usepackage{algorithmicx}%
\usepackage{algpseudocode}%
\usepackage{listings}%
\theoremstyle{thmstyleone}%
\newtheorem{theorem}{Theorem}%  meant for continuous numbers
\newtheorem{proposition}[theorem]{Proposition}% 

\theoremstyle{thmstyletwo}%
\newtheorem{example}{Example}%
\newtheorem{remark}{Remark}%

\theoremstyle{thmstylethree}%
\newtheorem{definition}{Definition}%
\newtheorem{corollary}[theorem]{Corollary}
\newtheorem{lemma}[theorem]{Lemma}
\begin{document}

\title[Birkhoff-Orthogonal Takahashi--von Neumann--Jordan Type Constants in Banach Spaces]{Birkhoff-Orthogonal Takahashi--von Neumann--Jordan Type Constants in Banach Spaces}


\author[1]{\fnm{Junxiang} \sur{Qi}}\email{y25060009@stu.aqnu.edu.cn}

\author[1]{\fnm{Zhouping} \sur{Yin}}\email{yzp@aqnu.edu.cn}
\equalcont{These authors contributed equally to this work.}

\author*[1]{\fnm{Qi} \sur{Liu}}\email{liuq67@aqnu.edu.cn}
\equalcont{These authors contributed equally to this work.}

\affil*[1]{\orgdiv{School of Mathematics and Statistics}, \orgname{Anqing Normal University}, \orgaddress{\city{Anqing}, \postcode{246133}, \country{P. R. China}}}


\abstract{In this paper, we introduce a new family of geometric constants \(C_t^B(X)\) of Birkhoff-orthogonal Takahashi-von Neumann-Jordan type for real Banach spaces. We discuss their basic properties and give a characterization of uniformly non-square spaces via \(C_t^B(X)\). Some sufficient conditions for super-reflexivity and fixed point property are presented. We establish upper bounds of \(C_t^B(X)\) related to orthogonal geometric moduli and constants, and obtain the sharp value for Radon planes. }




\maketitle

\section{Introduction}

Geometric constants of Banach spaces measure how far a norm is from the Euclidean norm.  Classical examples include the James non-square constant, the von Neumann--Jordan constant, the Zb\u{a}ganu constant, the modulus of convexity and the modulus of smoothness.  Their importance comes from the fact that analytic properties such as uniform non-squareness, normal structure, super-reflexivity, and fixed point properties can often be expressed by sharp numerical inequalities.Readers interested in further details may consult\cite{Clarkson1936,KatoMaligrandaTakahashi2001,TakahashiKato2007,BarontiCasiniPapini2000,Pisier2016,
	HeRaoWangLiuLi2026,LiuWang2026,AlonsoMartinPapini2026,IfronikaGunawanLindiarni2026}.

For a Banach space $X$ and $-\infty\leq t<\infty$, Takahashi introduced a scale of James type constants $J_{X,t}(\tau)$\cite{Takahashi2006}, $\tau\geq0$, and the associated von Neumann--Jordan type constant
\[
C_t(X)=\sup_{0\leq\tau\leq1}\frac{J_{X,t}(\tau)^2}{1+\tau^2}.
\]
This scale unifies several well-known quantities.  For instance, $J_{X,-\infty}(1)$ is the James constant\cite{James1964}, $C_2(X)$ is the usual von Neumann--Jordan constant\cite{2}, and $C_0(X)$ is the Zb\u{a}ganu constant\cite{3}.  A key feature of Takahashi's scale is the equivalence between uniform non-squareness and $C_t(X)<2$.

The present paper inserts Birkhoff-James orthogonality into this Takahashi scale.  Recall that for $x,y\in X$, one says that $x$ is Birkhoff-James orthogonal to $y$, written $x\perp_B y$\cite{Birkhoff1935}, if
\[
\|x+\lambda y\|\geq \|x\|\qquad (\lambda \in R).
\]
This notion is generally neither symmetric nor additive, but it is one of the most natural orthogonality relations in normed spaces.  In a Hilbert space it coincides with inner-product orthogonality.  Orthogonal versions of the James constant and related triangle constants have recently proved useful in the study of uniform non-squareness and Radon planes; see Baronti--Papini\cite{PapiniBaronti2022} and subsequent work by Du, Li \cite{DuLi2025}and collaborators.  Recent Heinz-mean constants\cite{PalChandok2025} under the Birkhoff restriction show that the interaction between means and orthogonality yields sharper geometric information than unrestricted constants.

Our main definition is the following.  For a Birkhoff-orthogonal pair $x,y\in S_X$, replace the two lengths $\|x+\tau y\|$ and $\|x-\tau y\|$ by their power mean.  Taking the supremum over all such orthogonal pairs gives $\mathsf{J}^{B}_{X,t}(\tau)$; normalizing by $1+\tau^2$ gives $\mathsf{C}^{B}_t(X)$.  Thus $\mathsf{C}^{B}_t(X)$ is an orthogonal, Takahashi-type, von Neumann--Jordan constant.  

This constant possesses three merits: it complies with the Euclidean Pythagorean normalization owing to the identity $\mathsf{C}^{B}_t(H)=1$ valid in Hilbert spaces, it sharply characterizes uniform non-squareness as $\mathsf{C}^{B}_t(X)$ equals exactly 2 for spaces failing to be uniformly non-square, and it enjoys great flexibility with respect to parameter $t$, which enables the transition from the minimum-type orthogonal James constant to constants defined via arithmetic, quadratic and geometric means.

The paper is organized as follows.  Section~\ref{sec:prelim} records basic facts on power means, Birkhoff-James orthogonality and orthogonal geometric constants.  Section~\ref{sec:def} introduces $\mathsf{J}^{B}_{X,t}(\tau)$ and $\mathsf{C}^{B}_t(X)$ and proves elementary properties.  Section~\ref{sec:radon} treats Radon planes and obtains the sharp value $5/4$ for affine regular hexagons when $t\leq1$. 

\section{Preliminaries}\label{sec:prelim}

Throughout the paper $X$ is a real Banach space with unit sphere $S_X$ and unit ball $B_X$.  Unless explicitly stated, $\dim X\geq2$.  For $a,b>0$ and $-\infty\leq t<\infty$, put
$$\mathsf{M}_t(a,b)=\left(\frac{a^t+b^t}{2}\right)^{1/t}.$$

For fixed positive real numbers \(a,b\in(0,\infty)\), we define a function \(\mathsf{M}_t\) mapping the extended real line \([-\infty,+\infty)\) into the positive interval \((0,\infty)\) as follows\cite{RaoLiu2026}:\[
\mathsf{M}_t(a,b)=
\begin{cases}
	\left(\dfrac{a^t+b^t}{2}\right)^{1/t}, & -\infty<t<\infty,\ t\neq0,\\[1.1em]
	\sqrt{ab}, & t=0,\\[0.4em]
	\min\{a,b\}, & t=-\infty.
\end{cases}
\]
The usual power-mean monotonicity says that if $-\infty\leq s\leq t<\infty$, then
\begin{equation}\label{eq:power-monotone}
	\mathsf{M}_s(a,b)\leq \mathsf{M}_t(a,b)\leq \max\{a,b\}.
\end{equation}
Moreover, $\mathsf{M}_t$ is increasing in each variable.

The following elementary observations will be used repeatedly.

\begin{lemma}\label{lem:anchor}
	Let $x,y\in S_X$ and $x\perp_B y$.  Then, for every $\tau\in[0,1]$,
	\begin{equation}\label{eq:anchor-bounds}
		1\leq \|x\pm\tau y\|\leq 1+\tau.
	\end{equation}
	Furthermore,
	\begin{equation}\label{eq:convex-interpolation}
		\|x\pm\tau y\|\leq (1-\tau)+\tau\|x\pm y\|.
	\end{equation}
\end{lemma}

\begin{proof}
	Since $x\perp_B y$, the inequality $\|x+\lambda y\|\geq1$ holds for every real $\lambda$, and hence for $\lambda=\pm\tau$.  The upper estimate in \eqref{eq:anchor-bounds} follows from the triangle inequality.  Finally,
	\[
	x\pm\tau y=(1-\tau)x+\tau(x\pm y),
	\]
	so \eqref{eq:convex-interpolation} follows from convexity of the norm.
\end{proof}

We shall use the following Birkhoff versions of two classical constants:

\begin{align}
	\mathsf{J}_{B}(X)&=\sup\{\min(\|x+y\|,\|x-y\|):x,y\in S_X, x\perp_B y\},\\
	\mathsf{A}_{2}^{B}(X)&=\sup\left\{\frac{\|x+y\|+\|x-y\|}{2}:x,y\in S_X,\ x\perp_B y\right\}.
\end{align}

The first one is the Birkhoff orthogonal James constant of Baronti and Papini\cite{PapiniBaronti2022}.  The second one is the Birkhoff orthogonal arithmetic triangle constant studied in connection with orthogonal moduli of convexity and smoothness\cite{DuLiangLi2025}.  Since $x \perp_B y$ implies $x\perp_B(-y)$, these constants are symmetric in the two signs $+y$ and $-y$, although Birkhoff orthogonality itself is not symmetric in $x$ and $y$.

The Birkhoff-orthogonal Takahashi-von Neumann-Jordan type constants investigated in this paper generalize the two classical constants mentioned above and Integrate the constants into the orthogonal framework.

A central fact, due to Papini and Baronti\cite{PapiniBaronti2022}, is
\begin{equation}\label{eq:PB-theorem}
	\mathsf{J}_{B}(X)=2\quad\Longleftrightarrow\quad X\text{ is not uniformly non-square}.
\end{equation}
Recall that $X$ is uniformly non-square if there exists $\delta>0$ such that for every $x,y\in S_X$ at least one of the inequalities \cite{James1964}
\[
\frac{\|x+y\|}{2}\leq 1-\delta,
\qquad
\frac{\|x-y\|}{2}\leq 1-\delta
\]
holds.

\section{The Birkhoff-Orthogonal Takahashi Constant}\label{sec:def}

We now introduce the main object of the paper.

\begin{definition}\label{def:new-constant}
	Let $X$ be a real Banach space, let $-\infty\leq t<\infty$, and let $\tau\geq0$.  The \emph{Birkhoff-orthogonal James type constant} of $X$ is
	\[
	\mathsf{J}^{B}_{X,t}(\tau)=\sup\left\{\mathsf{M}_t(\|x+\tau y\|,\|x-\tau y\|):x,y\in S_X,\ x\perp_B y\right\}.
	\]
	The associated \emph{Birkhoff-orthogonal Takahashi--von Neumann--Jordan type constant} is
	\begin{equation}\label{eq:CB-def}
		\mathsf{C}^{B}_t(X)=\sup_{0\leq\tau\leq1}\frac{\mathsf{J}^{B}_{X,t}(\tau)^2}{1+\tau^2}.
	\end{equation}
\end{definition}

\begin{remark}
	Specially,
	
	(i)If $t=1$, $\tau=1$,then
	$$\mathsf{J}^{B}_{X,1}(1)=\sup \left\{\frac{\|x+y\|+\|x-y\|}{2}: x, y \in S_X, x \perp_B y\right\}=\mathsf{A}_{2}^{B}(X) .$$
	
	(ii)If  \(t=-\infty,\tau=1\),then
	$$\mathsf{J}^{B}_{X,-\infty}(1)=\sup\{\min(\|x+y\|,\|x-y\|):x,y\in S_X, x\perp_B  y\}=\mathsf{J}_{B}(X)$$

\end{remark}
For $t=2$, the value $\mathsf{J}^{B}_{X,2}(1)^2$ is the supremum of the orthogonal parallelogram expression
\[
\frac{\|x+y\|^2+\|x-y\|^2}{2},\qquad x,y\in S_X,\\ x\perp_B  y.
\]
Thus $\mathsf{C}^{B}_2(X)$ may be interpreted as a Birkhoff-orthogonal von Neumann--Jordan type constant.  For $t=0$ it is a Birkhoff-orthogonal Zb\u{a}ganu-type constant, and for $t=-\infty$ it is governed by the minimum of the two orthogonal chord lengths.

\begin{proposition}\label{prop:basic}
	For every real Banach space $X$ and every $-\infty\leq t<\infty$,
	\begin{equation}\label{eq:basic-bounds}
		1\leq \mathsf{C}^{B}_t(X)\leq 2.
	\end{equation}
	Moreover:
	\begin{enumerate}
		\item $\mathsf{J}^{B}_{X,t}(0)=1$ and $1\leq \mathsf{J}^{B}_{X,t}(\tau)\leq1+\tau$ for $0\leq\tau\leq1$;
		\item if $-\infty\leq s\leq t<\infty$, then $\mathsf{J}^{B}_{X,s}(\tau)\leq \mathsf{J}^{B}_{X,t}(\tau)$ for every $\tau\geq0$, and consequently $\mathsf{C}^{B}_s(X)\leq\mathsf{C}^{B}_t(X)$;
		\item if $Y$ is a closed subspace of $X$, then $\mathsf{C}^{B}_t(Y)\leq\mathsf{C}^{B}_t(X)$;
		\item $\mathsf{C}^{B}_t(X)\leq C_t(X)$, where $C_t(X)$ is Takahashi's unrestricted constant.
	\end{enumerate}
\end{proposition}

\begin{proof}
	Part (a) follows immediately from  Lemma \ref{lem:anchor} and the monotonicity of $\mathsf{M}_t$ in each variable.  Since $\tau=0$ is allowed in \eqref{eq:CB-def}, the lower bound $\mathsf{C}^{B}_t(X)\geq1$ follows.  The upper bound follows from
	\[
	\frac{\mathsf{J}^{B}_{X,t}(\tau)^2}{1+\tau^2}\leq \frac{(1+\tau)^2}{1+\tau^2}\leq2,
	\]
	where equality in the last scalar inequality is possible only at $\tau=1$.
	
	Part (b) is exactly the monotonicity of power means.  For (c), observe that if $x,y\in S_Y$ and $x\perp_B y$ in $Y$, then the same inequality $\|x+\lambda y\|\geq\|x\|$ holds in $X$, because $Y$ carries the inherited norm.  Thus the admissible pairs for $Y$ are among the admissible pairs for $X$.  Finally, (d) follows because the Birkhoff-orthogonal pairs form a subset of all pairs in $S_X\times S_X$.
\end{proof}

\begin{proposition}\label{prop:hilbert}
	If $H$ is a Hilbert space, then
	\[
	\mathsf{J}^{B}_{H,t}(\tau)=\sqrt{1+\tau^2}
	\quad (\tau\geq0,
	\ -\infty\leq t<\infty),
	\]
	and therefore
	\[
	\mathsf{C}^{B}_t(H)=1.
	\]
\end{proposition}

\begin{proof}
	In a Hilbert space, $x\perp_B y$ is equivalent to $\langle x,y\rangle=0$.  Hence, for $x,y\in S_H$ with $x\perp_B y$,
	\[
	\|x+\tau y\|=\|x-\tau y\|=\sqrt{1+\tau^2}.
	\]
	All power means of two equal positive numbers are equal to that common value.  The conclusion follows from the definition.
\end{proof}

The preceding result explains the normalization in \eqref{eq:CB-def}.  The constant is equal to one in Hilbert spaces, while its maximum possible value is two.

\begin{example}\label{ex:linfty2}
	Let $X=\ell_\infty^2$.  Put $x=(1,1)$ and $y=(-1,1)$.  Then $x,y\in S_X$ and
	\[
	\|x+\lambda y\|_\infty=\max\{|1-\lambda|,|1+\lambda|\}\geq1,
	\]
	so $x\perp_B y$.  For $0\leq\tau\leq1$,
	\[
	\|x+\tau y\|_\infty=\|x-\tau y\|_\infty=1+\tau.
	\]
	Together with Proposition \ref{prop:basic}, this gives
	\[
	\mathsf{J}^{B}_{\ell_\infty^2,t}(\tau)=1+\tau,
	\qquad
	\mathsf{C}^{B}_t(\ell_\infty^2)=2
	\]
	for every $-\infty\leq t<\infty$.
\end{example}

\begin{example}\label{thm:linfty-sum}
	Let $Y$ and $Z$ be nonzero Banach spaces and let $X=Y\oplus_\infty Z$.  Then
	\[
	\mathsf{C}^{B}_t(X)=2
	\qquad(-\infty\leq t<\infty).
	\]
\end{example}

\begin{proof}
	Choose $u\in S_Y$ and $v\in S_Z$.  Put $x=(u,v)$ and $y=(u,-v)$ in $X$.  Then $x,y\in S_X$ and for every $\lambda\in R$,
	\[
	\|x+\lambda y\|=\max\{|1+\lambda|\|u\|,|1-\lambda|\|v\|\}
	=\max\{|1+\lambda|,|1-\lambda|\}\geq1.
	\]
	Thus $x\perp_B y$.  For $0\leq\tau\leq1$,
	\[
	\|x+\tau y\|=\|x-\tau y\|=1+\tau.
	\]
	It follows that $\mathsf{J}^{B}_{X,t}(\tau)=1+\tau$ and hence $\mathsf{C}^{B}_t(X)=2$ by Proposition \ref{prop:basic}.
\end{proof}

\begin{corollary}
	Every nontrivial $\ell_\infty$-sum is not uniformly non-square.
\end{corollary}

Next we will prove that the endpoint value $2$ of the new constant is exactly the obstruction to uniform non-squareness.

We first record a compactness-free saturation lemma for power means.

\begin{lemma}\label{lem:saturation}
	Fix $-\infty\leq t<\infty$.  Let $A_n>0$ and $0<a_n,b_n\leq A_n$.  If
	\[
	\frac{\mathsf{M}_t(a_n,b_n)}{A_n}\longrightarrow1,
	\]
	then
	\[
	\frac{\min\{a_n,b_n\}}{A_n}\longrightarrow1.
	\]
\end{lemma}

\begin{proof}
	For $t=-\infty$ the assertion is immediate.  Let $t$ be finite.  Suppose the conclusion fails.  Then there are $\varepsilon>0$ and a subsequence, still denoted by $n$, such that
	\[
	\min\{a_n,b_n\}\leq(1-\varepsilon)A_n.
	\]
	Set $u_n=a_n/A_n$ and $v_n=b_n/A_n$.  Then $0<u_n,v_n\leq1$ and $\min\{u_n,v_n\}\leq1-\varepsilon$.  Since $\mathsf{M}_t$ is increasing in each variable,
	\[
	\mathsf{M}_t(u_n,v_n)\leq \mathsf{M}_t(1,1-\varepsilon)<1.
	\]
	This contradicts $\mathsf{M}_t(a_n,b_n)/A_n=\mathsf{M}_t(u_n,v_n)\to1$.
\end{proof}

\begin{theorem}\label{thm:endpoint}
	Let $X$ be a real Banach space and $-\infty\leq t<\infty$.  Then
	\begin{equation}\label{eq:endpoint-equivalence}
		\mathsf{C}^{B}_t(X)=2
		\quad\Longleftrightarrow\quad
		X\text{ is not uniformly non-square}.
	\end{equation}
	Consequently,
	\[
	\mathsf{C}^{B}_t(X)<2
	\quad\Longleftrightarrow\quad
	X\text{ is uniformly non-square}.
	\]
\end{theorem}

\begin{proof}
	Assume first that $X$ is not uniformly non-square.  By \eqref{eq:PB-theorem}, $\mathsf{J}_{B}(X)=2$.  Thus, for every $\varepsilon>0$, there exist $x_\varepsilon,y_\varepsilon\in S_X$ with $x_\varepsilon\perp_B y_\varepsilon$ and
	\[
	\min\{\|x_\varepsilon+y_\varepsilon\|,\|x_\varepsilon-y_\varepsilon\|\}>2-\varepsilon.
	\]
	Since both norms are at most $2$, the monotonicity of $\mathsf{M}_t$ gives
	\[
	\mathsf{J}^{B}_{X,t}(1)\geq \mathsf{M}_t(2-\varepsilon,2-\varepsilon)=2-\varepsilon.
	\]
	Hence
	\[
	\mathsf{C}^{B}_t(X)\geq \frac{(2-\varepsilon)^2}{2}.
	\]
	Letting $\varepsilon\downarrow0$ and using Proposition \ref{prop:basic}, we get $\mathsf{C}^{B}_t(X)=2$.
	
	Conversely, assume that $\mathsf{C}^{B}_t(X)=2$.  Choose $\tau_n\in[0,1]$ and $x_n,y_n\in S_X$ with $x_n\perp_B y_n$ such that, with
	\[
	a_n=\|x_n+\tau_n y_n\|,
	\qquad
	b_n=\|x_n-\tau_n y_n\|,
	\]
	we have
	\begin{equation}\label{eq:approach-two}
		\frac{\mathsf{M}_t(a_n,b_n)^2}{1+\tau_n^2}\longrightarrow2.
	\end{equation}
	By Lemma \ref{lem:anchor}, $a_n,b_n\leq1+\tau_n$, so
	\[
	\frac{\mathsf{M}_t(a_n,b_n)^2}{1+\tau_n^2}
	\leq
	\frac{(1+\tau_n)^2}{1+\tau_n^2}
	\leq2.
	\]
	The scalar function $(1+\tau)^2/(1+\tau^2)$ has the unique maximum $2$ on $[0,1]$ at $\tau=1$.  Hence \eqref{eq:approach-two} implies $\tau_n\to1$ and
	\[
	\frac{\mathsf{M}_t(a_n,b_n)}{1+\tau_n}\longrightarrow1.
	\]
	Applying Lemma \ref{lem:saturation} with $A_n=1+\tau_n$, we get
	\[
	\min\{a_n,b_n\}\longrightarrow2.
	\]
	It follows that $\mathsf{J}_{B}(X)=2$.  By \eqref{eq:PB-theorem}, $X$ is not uniformly non-square.
\end{proof}

\begin{corollary}\label{cor:superreflexive}
	If $\mathsf{C}^{B}_t(X)<2$ for one, equivalently for every, $-\infty\leq t<\infty$, then $X$ is super-reflexive and has the fixed point property for nonexpansive self-mappings on closed bounded convex subsets.
\end{corollary}

\begin{proof}
	By Theorem \ref{thm:endpoint}, $X$ is uniformly non-square.  The conclusion follows from the standard consequences of uniform non-squareness.
\end{proof}

\begin{remark}
	The equivalence in Theorem \ref{thm:endpoint} is stronger than a formal restriction of Takahashi's unrestricted result.  The Birkhoff condition could have destroyed extremal pairs, but the Papini--Baronti theorem \eqref{eq:PB-theorem} shows that, at the endpoint $2$, enough extremal pairs remain Birkhoff orthogonal.
\end{remark}

We next compare $\mathsf{C}^{B}_t(X)$ with the Birkhoff James constant $\mathsf{J}_{B}(X)$ and the Birkhoff arithmetic triangle constant $\mathsf{A}_{2}^{B}(X)$.

Let
\[
j=\mathsf{J}_{B}(X).
\]
For $0\leq\tau\leq1$ set
\[
\Phi_{t,j}(\tau)=\mathsf{M}_t(1+\tau,1-\tau+\tau j).
\]

\begin{theorem}\label{thm:J-bound}\label{prop:JBcompare}
	For every Banach space $X$ and every $-\infty\leq t<\infty$,
	\begin{equation}\label{eq:J-bound}
		\mathsf{J}^{B}_{X,t}(\tau)\leq \Phi_{t,\mathsf{J}_{B}(X)}(\tau)
		\qquad(0\leq\tau\leq1).
	\end{equation}
	Consequently,
	\begin{equation}\label{eq:CJ-bound}
		\mathsf{C}^{B}_t(X)\leq \sup_{0\leq\tau\leq1}
		\frac{\Phi_{t,\mathsf{J}_{B}(X)}(\tau)^2}{1+\tau^2}.
	\end{equation}
\end{theorem}

\begin{proof}
	Let $x,y\in S_X$ and $x\perp_B y$.  Put
	\[
	A=\|x+\tau y\|,
	\qquad
	B=\|x-\tau y\|.
	\]
	By Lemma \ref{lem:anchor}, $A,B\leq1+\tau$.  On the other hand, using \eqref{eq:convex-interpolation},
	\[
	\min\{A,B\}
	\leq
	1-\tau+\tau\min\{\|x+y\|,\|x-y\|\}
	\leq 1-\tau+\tau\mathsf{J}_{B}(X).
	\]
	Since $\mathsf{M}_t$ is increasing in each variable, the largest possible value under these two constraints is obtained by placing one variable at $1+\tau$ and the other at $1-\tau+\tau\mathsf{J}_{B}(X)$.  Thus
	\[
	\mathsf{M}_t(A,B)\leq \mathsf{M}_t(1+\tau,1-\tau+\tau\mathsf{J}_{B}(X)).
	\]
	Taking the supremum over all admissible pairs gives \eqref{eq:J-bound}; then \eqref{eq:CJ-bound} follows from the definition of $\mathsf{C}^{B}_t(X)$.
\end{proof}

For $t=-\infty$ this becomes an explicit two-sided estimate.

\begin{corollary}\label{cor:min-bound}
	For every Banach space $X$,
	\begin{equation}\label{eq:min-two-sided}
		\frac{\mathsf{J}_{B}(X)^2}{2}\leq \mathsf{C}^{B}_{-\infty}(X)
		\leq 1+\big(\mathsf{J}_{B}(X)-1\big)^2.
	\end{equation}
	In particular, if $\mathsf{J}_{B}(X)<2$, then $\mathsf{C}^{B}_{-\infty}(X)<2$.
\end{corollary}

\begin{proof}
	The lower bound follows by taking $\tau=1$ in the definition:
	\[
	\mathsf{C}^{B}_{-\infty}(X)
	\geq \frac{\mathsf{J}^{B}_{X,-\infty}(1)^2}{2}
	=\frac{\mathsf{J}_{B}(X)^2}{2}.
	\]
	For the upper bound, Theorem \ref{thm:J-bound} gives
	\[
	\mathsf{J}^{B}_{X,-\infty}(\tau)
	\leq 1-\tau+\tau\mathsf{J}_{B}(X)=1+a\tau,
	\qquad a=\mathsf{J}_{B}(X)-1\in[0,1].
	\]
	The scalar maximum
	\[
	\max_{0\leq\tau\leq1}\frac{(1+a\tau)^2}{1+\tau^2}=1+a^2
	\]
	is obtained at $\tau=a$.  This proves the upper bound.
\end{proof}

The previous estimate uses the smaller chord.  For $t\leq1$, a sharper and more symmetric estimate is possible because power means do not exceed the arithmetic mean.

\begin{theorem}\label{thm:A-bound}
	Let $-\infty\leq t\leq1$.  Then, for every Banach space $X$,
	\begin{equation}\label{eq:A-tau-bound}
		\mathsf{J}^{B}_{X,t}(\tau)\leq 1+\tau\big(\mathsf{A}_{2}^{B}(X)-1\big)
		\qquad(0\leq\tau\leq1),
	\end{equation}
	and therefore
	\begin{equation}\label{eq:A-C-bound}
		\mathsf{C}^{B}_t(X)\leq 1+\big(\mathsf{A}_{2}^{B}(X)-1\big)^2.
	\end{equation}
\end{theorem}

\begin{proof}
	Let $x,y\in S_X$ and $x\perp_B y$.  Since $t\leq1$, $\mathsf{M}_t(a,b)\leq(a+b)/2$ for all $a,b>0$.  Hence, by \eqref{eq:convex-interpolation},
	\begin{align*}
		\mathsf{M}_t(\|x+\tau y\|,\|x-\tau y\|)
		&\leq \frac{\|x+\tau y\|+\|x-\tau y\|}{2}\\
		&\leq 1-\tau+\tau\frac{\|x+y\|+\|x-y\|}{2}\\
		&\leq 1+\tau(\mathsf{A}_{2}^{B}(X)-1).
	\end{align*}
	Taking suprema gives \eqref{eq:A-tau-bound}.  Put $a=\mathsf{A}_{2}^{B}(X)-1\in[0,1]$.  Then
	\[
	\mathsf{C}^{B}_t(X)\leq \sup_{0\leq\tau\leq1}\frac{(1+a\tau)^2}{1+\tau^2}=1+a^2,
	\]
	which is \eqref{eq:A-C-bound}.
\end{proof}

\begin{remark}
	The estimate \eqref{eq:A-C-bound} is sharp in the following two very different regimes.  In $\ell_\infty^2$, one has $\mathsf{A}_{2}^{B}(X)=2$ and $\mathsf{C}^{B}_t(X)=2$.  In affine regular hexagonal Radon planes, $\mathsf{A}_{2}^{B}(X)=3/2$ and, for $t\leq1$, $\mathsf{C}^{B}_t(X)=5/4$. Thus the scalar factor $1+(\mathsf{A}_{2}^{B}(X)-1)^2$ cannot be improved in this generality.
\end{remark}

We now relate the new constant to orthogonal moduli of convexity and smoothness.  Several variants of these moduli exist; the following forms are tailored to the present estimates.

Define\cite{DuLi2025}
\[
\delta_B^+(X)=\inf\left\{1-\frac{\|x+y\|}{2}:x,y\in S_X,\ x\perp_B y\right\}.
\]
Since $x\perp_B y$ implies $x\perp_B(-y)$, the same bound applies to both signs.  Thus $0\leq\delta_B^+(X)\leq1/2$.

\begin{proposition}\label{prop:delta-bound}
	For every Banach space $X$ and every $-\infty\leq t<\infty$,
	\begin{equation}\label{eq:delta-bound}
		\mathsf{C}^{B}_t(X)\leq 1+\big(1-2\delta_B^+(X)\big)^2.
	\end{equation}
\end{proposition}

\begin{proof}
	Let $x,y\in S_X$ with $x\perp_B y$.  By definition of $\delta_B^+(X)$ and by applying it to $y$ and $-y$,
	\[
	\|x+y\|\leq2(1-\delta_B^+(X)),
	\qquad
	\|x-y\|\leq2(1-\delta_B^+(X)).
	\]
	Using \eqref{eq:convex-interpolation},
	\[
	\|x\pm\tau y\|\leq1-\tau+2\tau(1-\delta_B^+(X))=1+\tau(1-2\delta_B^+(X)).
	\]
	Therefore $\mathsf{J}^{B}_{X,t}(\tau)\leq1+\tau(1-2\delta_B^+(X))$.  Maximizing the normalized scalar expression as in the proof of Theorem \ref{thm:A-bound} gives \eqref{eq:delta-bound}.
\end{proof}

The one-sided Birkhoff smoothness envelope is defined as follows \cite{1}:
\[
\rho_B(\tau;X)=\sup\left\{\frac{\|x+\tau y\|+\|x-\tau y\|}{2}-1:x,y\in S_X,\ x\perp_B y\right\}.
\]
And $0\leq\rho_B(\tau;X)\leq\tau$.

\begin{proposition}\label{prop:rho-bound}
	If $-\infty\leq t\leq1$, then
	\begin{equation}\label{eq:rho-bound}
		\mathsf{C}^{B}_t(X)
		\leq
		\sup_{0\leq\tau\leq1}\frac{\big(1+\rho_B(\tau;X)\big)^2}{1+\tau^2}.
	\end{equation}
	In particular, if $\rho_B(\tau;X)\leq L\tau$ for $0\leq\tau\leq1$ with $0\leq L\leq1$, then
	\begin{equation}\label{eq:rho-linear}
		\mathsf{C}^{B}_t(X)\leq 1+L^2.
	\end{equation}
\end{proposition}

\begin{proof}
	For $t\leq1$, the power mean is bounded above by the arithmetic mean.  Taking the supremum over Birkhoff-orthogonal pairs gives
	\[
	\mathsf{J}^{B}_{X,t}(\tau)\leq1+\rho_B(\tau;X),
	\]
	which proves \eqref{eq:rho-bound}.  If $\rho_B(\tau;X)\leq L\tau$, the same scalar maximization yields \eqref{eq:rho-linear}.
\end{proof}

For $t=2$ we can use quadratic smoothness.  Define the quadratic smoothness coefficient
\[
\mathfrak{s}_2(X)=\sup_{u\in S_X,\ v\neq0}
\frac{\|u+v\|^2+\|u-v\|^2-2}{2\|v\|^2}
\]
with the convention that it may be $+\infty$.  Hilbert spaces have $\mathfrak{s}_2(H)=1$.

\begin{proposition}\label{prop:quadratic-smooth}
	If $\mathfrak{s}_2(X)<\infty$, then
	\begin{equation}\label{eq:J2-smooth}
		\mathsf{J}^{B}_{X,2}(\tau)^2\leq1+\mathfrak{s}_2(X)\tau^2
		\qquad(0\leq\tau\leq1),
	\end{equation}
	and
	\begin{equation}\label{eq:C2-smooth}
		\mathsf{C}^{B}_2(X)\leq
		\begin{cases}
			1, & \mathfrak{s}_2(X)\leq1,\\[0.3em]
			\dfrac{1+\mathfrak{s}_2(X)}{2}, & \mathfrak{s}_2(X)\geq1.
		\end{cases}
	\end{equation}
	Consequently, for $t\leq2$, the same upper bound holds for $\mathsf{C}^{B}_t(X)$.
\end{proposition}

\begin{proof}
	For $x,y\in S_X$ and $x\perp_B y$, apply the definition of $\mathfrak{s}_2(X)$ with $u=x$ and $v=\tau y$.  Then
	\[
	\frac{\|x+\tau y\|^2+\|x-\tau y\|^2}{2}\leq1+\mathfrak{s}_2(X)\tau^2.
	\]
	Taking suprema gives \eqref{eq:J2-smooth}.  Therefore
	\[
	\mathsf{C}^{B}_2(X)\leq\sup_{0\leq\tau\leq1}\frac{1+\mathfrak{s}_2(X)\tau^2}{1+\tau^2}.
	\]
	The last scalar function is decreasing when $\mathfrak{s}_2(X)\leq1$ and increasing when $\mathfrak{s}_2(X)\geq1$, which gives \eqref{eq:C2-smooth}.  The final assertion follows from the monotonicity of $\mathsf{C}^{B}_t(X)$ in $t$.
\end{proof}

\begin{remark}
	In a smooth Banach space, $x\perp_B y$ is equivalent to $j(x)(y)=0$, where $j(x)$ is the unique norming functional at $x$.  Hence $\mathsf{J}^{B}_{X,2}(\tau)^2$ measures the largest second-order symmetric deviation of the norm along tangent directions to the unit sphere.  This gives a differential-geometric interpretation of the orthogonal von Neumann--Jordan type constant.
\end{remark}

The next theorem shows that uniform convexity alone does not force a uniform gap of $\mathsf{C}^{B}_t(X)$ from $2$.

\begin{theorem}\label{thm:lp-lower}
	Let $1<p<\infty$ and $X=\ell_p^2$.  Then, for every $-\infty\leq t<\infty$,
	\begin{equation}\label{eq:lp-lower}
		\mathsf{C}^{B}_t(\ell_p^2)
		\geq
		\sup_{0\leq\tau\leq1}
		\frac{q_p(\tau)^2}{1+\tau^2},
		\qquad
		q_p(\tau)=\left(\frac{(1+\tau)^p+(1-\tau)^p}{2}\right)^{1/p}.
	\end{equation}
	In particular,
	\begin{equation}\label{eq:lp-limit}
		\lim_{p\to\infty}\mathsf{C}^{B}_t(\ell_p^2)=2
	\end{equation}
	for every fixed $t$.
\end{theorem}

\begin{proof}
	Let
	\[
	x=(2^{-1/p},2^{-1/p}),
	\qquad
	y=(-2^{-1/p},2^{-1/p}).
	\]
	Then $x,y\in S_{\ell_p^2}$.  Since $\ell_p^2$ is smooth for $1<p<\infty$, $x\perp_B y$ is equivalent to
	\[
	|x_1|^{p-2}x_1y_1+|x_2|^{p-2}x_2y_2=0,
	\]
	which is immediate for this pair.  Moreover,
	\[
	\|x+\tau y\|_p=\|x-\tau y\|_p
	=\left(\frac{(1+\tau)^p+(1-\tau)^p}{2}\right)^{1/p}=q_p(\tau).
	\]
	Therefore $\mathsf{J}^{B}_{\ell_p^2,t}(\tau)\geq q_p(\tau)$, giving \eqref{eq:lp-lower}.  As $p\to\infty$, $q_p(\tau)\to1+\tau$ uniformly on compact subintervals of $[0,1]$ and pointwise on $[0,1]$.  Hence the right-hand side of \eqref{eq:lp-lower} tends to
	\[
	\sup_{0\leq\tau\leq1}\frac{(1+\tau)^2}{1+\tau^2}=2.
	\]
	Together with the universal upper bound $\mathsf{C}^{B}_t(\ell_p^2)\leq2$, this proves \eqref{eq:lp-limit}.
\end{proof}

\begin{corollary}\label{cor:uc-near-two}
	For every $\varepsilon>0$ and every $-\infty\leq t<\infty$, there exists a finite-dimensional uniformly convex Banach space $X$ such that
	\[
	\mathsf{C}^{B}_t(X)>2-\varepsilon.
	\]
\end{corollary}

\begin{proof}
	Take $X=\ell_p^2$ with $p$ sufficiently large.  The space $\ell_p^2$ is uniformly convex for $1<p<\infty$, and the assertion follows from Theorem \ref{thm:lp-lower}.
\end{proof}

\begin{remark}
	The preceding corollary shows that Theorem \ref{thm:endpoint} is qualitative rather than uniformly quantitative over all uniformly convex spaces.  A quantitative gap below $2$ requires a specified modulus, a specified value of $\mathsf{J}_{B}(X)$, or a structural assumption such as a Radon-plane condition.
\end{remark}

\section{Radon Planes and the Sharp Constant $5/4$}\label{sec:radon}

A two-dimensional normed space is called a Radon plane if Birkhoff-James orthogonality is symmetric.  Radon planes form the correct two-dimensional substitute for the Hilbert-space symmetry of orthogonality.  For the Birkhoff arithmetic triangle constant it is known that
\begin{equation}\label{eq:Radon-A}
	\mathsf{A}_{2}^{B}(X)\leq\frac32
\end{equation}
for every Radon plane $X$, and equality holds if and only if the unit sphere is an affine regular hexagon.

Combining \eqref{eq:Radon-A} with Theorem \ref{thm:A-bound} gives the following sharp theorem.

\begin{theorem}\label{thm:radon-sharp}
	Let $X$ be a Radon plane and let $-\infty\leq t\leq1$.  Then
	\begin{equation}\label{eq:radon-bound}
		\mathsf{C}^{B}_t(X)\leq\frac54.
	\end{equation}
	Moreover,
	\begin{equation}\label{eq:radon-eq}
		\mathsf{C}^{B}_t(X)=\frac54
		\quad\Longleftrightarrow\quad
		S_X\text{ is an affine regular hexagon}.
	\end{equation}
\end{theorem}

\begin{proof}
	By Theorem \ref{thm:A-bound} and \eqref{eq:Radon-A},
	\[
	\mathsf{C}^{B}_t(X)
	\leq1+\left(\frac32-1\right)^2
	=\frac54.
	\]
	If equality holds, then the estimate above forces $\mathsf{A}_{2}^{B}(X)=3/2$.  By the known characterization of equality in \eqref{eq:Radon-A}, the unit sphere is an affine regular hexagon.
	
	Conversely, suppose that $S_X$ is an affine regular hexagon.  After an affine change of coordinates, we may assume that
	\[
	S_X=\partial\operatorname{conv}\{\pm u,\pm v,\pm(u+v)\}.
	\]
	In coordinates $\alpha u+\beta v$, the norm is
	\begin{equation}\label{eq:hex-norm}
		\|\alpha u+\beta v\|=\max\{|\alpha|,|\beta|,|\alpha-\beta|\}.
	\end{equation}
	Take
	\[
	x=v,
	\qquad
	y=u+\frac12v.
	\]
	Then $x,y\in S_X$.  For every $\lambda\in R$, using \eqref{eq:hex-norm},
	\[
	\|x+\lambda y\|
	=\left\|\lambda u+\left(1+\frac\lambda2\right)v\right\|
	=\max\left\{|\lambda|,\left|1+\frac\lambda2\right|,\left|1-\frac\lambda2\right|\right\}\geq1.
	\]
	Thus $x\perp_B y$.  For $0\leq\tau\leq1$,
	\begin{align*}
		\|x+\tau y\|
		&=\left\|\tau u+\left(1+\frac\tau2\right)v\right\|=1+\frac\tau2,\\
		\|x-\tau y\|
		&=\left\|-\tau u+\left(1-\frac\tau2\right)v\right\|=1+\frac\tau2.
	\end{align*}
	Therefore
	\[
	\mathsf{J}^{B}_{X,t}(\tau)\geq1+\frac\tau2.
	\]
	At $\tau=1/2$ we obtain
	\[
	\mathsf{C}^{B}_t(X)
	\geq
	\frac{(1+1/4)^2}{1+1/4}
	=\frac54.
	\]
	Together with the upper bound this proves equality.
\end{proof}

\begin{remark}
	The number $5/4$ is the same scalar threshold appearing in several Takahashi-type estimates for the unrestricted constant $C_1(X)$.  Here it appears for a different reason: the Birkhoff orthogonality restriction and the Radon-plane inequality $\mathsf{A}_{2}^{B}(X)\leq3/2$ reduce the extremal problem to the one-variable maximization of $(1+\tau/2)^2/(1+\tau^2)$, whose maximum occurs at $\tau=1/2$.
\end{remark}

The maximal value theorem is qualitative, but the proof also gives an explicit nonsquareness estimate. Define the Birkhoff nonsquareness defect
\[
\eta_B(X)=1-\frac{J_B(X)}{2}.
\]
Thus $\eta_B(X)>0$ is equivalent to uniform non-squareness.

\begin{theorem}\label{thm:quant}
	For every Banach space $X$ and every $-\infty\le t<\infty$,
	\begin{equation}\label{eq:quant}
		\eta_B(X)\ge 1-\sqrt{\frac{C_t^B(X)}{2}}.
	\end{equation}
	Equivalently, if $C_t^B(X)\le2-\delta$ for some $0<\delta<2$, then
	\begin{equation}\label{eq:quant-delta}
		J_B(X)\le \sqrt{4-2\delta}
		\quad\text{and}\quad
		\eta_B(X)\ge 1-\sqrt{1-\delta/2}.
	\end{equation}
\end{theorem}

\begin{proof}
	By Proposition \ref{prop:JBcompare},
	\[
	\frac{J_B(X)^2}{2}\le C_t^B(X).
	\]
	Therefore $J_B(X)\le\sqrt{2C_t^B(X)}$, and division by two gives \eqref{eq:quant}. Formula \eqref{eq:quant-delta} follows by substituting $C_t^B(X)\le2-\delta$.
\end{proof}

\begin{corollary}\label{cor:qfp}
	If $C_t^B(X)\le2-\delta$ for some $\delta>0$, then all Birkhoff-orthogonal almost-squares in $X$ are separated from the extremal square by at least
	\[
	1-\sqrt{1-\delta/2}.
	\]
	
\end{corollary}

\begin{proof}
	The first assertion is just \eqref{eq:quant-delta}. 
\end{proof}

In the unrestricted setting, the modulus of convexity gives exact formulas for several James type constants. Here the orthogonality condition requires an adapted profile.

Following \cite{4}, set
\[
D_B(X)=\{\|x-y\|:x,y\in S_X,\ x\perp_B y\}\subset[1,2].
\]

The classical Clarkson convex modulus \cite{Clarkson1936}
$$\delta_X(r) = \inf\left\{1-\frac{\|x+y\|}{2}:x,y\in S_X,\ \|x-y\|=r\right\},\quad 0\le r\le 2.$$
Next, we introduce the following definition, which will be applied in the theorem below:
for $r\in D_B(X)$ define
\begin{equation}\label{eq:DeltaB}
	\Delta_B(r)=\inf\left\{1-\frac{\|x+y\|}{2}:x,y\in S_X,\ x\perp_B y,
	\ \|x-y\|=r\right\}.
\end{equation}
We call $\Delta_B$ the Birkhoff-orthogonal convexity profile of $X$.

The profile $\Delta_B$ is not a modulus on a full interval in general, because the set $D_B(X)$ may fail to be an interval. Nevertheless it is exactly the object needed for $J_{X,t}^{B}(1)$.

\begin{theorem}\label{thm:profile}
	For every Banach space $X$ and every $-\infty\le t<\infty$,
	\begin{equation}\label{eq:profileformula}
		J_{X,t}^{B}(1)=
		\sup_{r\in D_B(X)} M_t\bigl(r,2(1-\Delta_B(r))\bigr).
	\end{equation}
\end{theorem}

\begin{proof}
	Let $x,y\in S_X$ with $x\perp_B y$, and put $r=\|x-y\|\in D_B(X)$. By definition of $\Delta_B(r)$,
	\[
	1-\frac{\|x+y\|}{2}\ge \Delta_B(r),
	\]
	so
	\[
	\|x+y\|\le2(1-\Delta_B(r)).
	\]
	Since $M_t$ is increasing in each variable,
	\[
	M_t(\|x-y\|,\|x+y\|)
	\le M_t(r,2(1-\Delta_B(r))).
	\]
	Taking the supremum over all Birkhoff-orthogonal pairs gives
	\[
	J_{X,t}^{B}(1)\le \sup_{r\in D_B(X)}M_t(r,2(1-\Delta_B(r))).
	\]
	Conversely, fix $r\in D_B(X)$ and $\eta>0$. By the definition of the infimum, there exist $x,y\in S_X$ with $x\perp_B y$, $\|x-y\|=r$, and
	\[
	1-\frac{\|x+y\|}{2}<\Delta_B(r)+\eta.
	\]
	Equivalently,
	\[
	\|x+y\|>2(1-\Delta_B(r))-2\eta.
	\]
	Thus
	\[
	J_{X,t}^{B}(1)
	\ge M_t\bigl(r,2(1-\Delta_B(r))-2\eta\bigr).
	\]
	Letting $\eta\downarrow0$ and then taking the supremum over $r\in D_B(X)$ proves the reverse inequality.
\end{proof}

\begin{corollary}\label{cor:profile_uns}
	The following assertions are equivalent:
	\begin{enumerate}
		\item $X$ is uniformly non-square.
		\item There exists $\eta>0$ such that
		\[
		M_t\bigl(r,2(1-\Delta_B(r))\bigr)\le2-\eta
		\qquad(r\in D_B(X))
		\]
		for one, equivalently for every, $-\infty\le t<\infty$.
		\item $\sup_{r\in D_B(X)}\min\{r,2(1-\Delta_B(r))\}<2$.
	\end{enumerate}
\end{corollary}

\begin{proof}
	By Theorem \ref{thm:profile} with $t=-\infty$,
	\[
	J_{X,-\infty}^{B}(1)=\sup_{r\in D_B(X)}\min\{r,2(1-\Delta_B(r))\}.
	\]
	But $J_{X,-\infty}^{B}(1)=J_B(X)$, and $J_B(X)<2$ is equivalent to uniform non-squareness by \eqref{eq:PB-theorem}. The equivalence with (ii) follows from the monotonicity of $M_t$ and the maximal value theorem.
\end{proof}

Next we collect several further observations that may be useful in applications.

\begin{proposition}\label{prop:endpoint-forms}
	For a Banach space $X$ and $-\infty\leq t<\infty$, the following are equivalent:
	\begin{enumerate}
		\item $X$ is uniformly non-square;
		\item $\mathsf{C}^{B}_t(X)<2$;
		\item there exists $\tau_0\in(0,1]$ such that
		\[
		\mathsf{J}^{B}_{X,t}(\tau_0)<1+\tau_0;
		\]
		\item for every $\tau\in(0,1]$,
		\[
		\mathsf{J}^{B}_{X,t}(\tau)<1+\tau.
		\]
	\end{enumerate}
\end{proposition}

\begin{proof}
	The equivalence of (i) and (ii) is Theorem \ref{thm:endpoint}.  Clearly (iv) implies (iii).  If (iii) fails, then $\mathsf{J}^{B}_{X,t}(\tau)=1+\tau$ for every $\tau\in(0,1]$, and hence $\mathsf{C}^{B}_t(X)=2$ by taking $\tau=1$.  Thus (ii) implies (iii).
	
	It remains to show that (iii) implies (i).  If $X$ were not uniformly non-square, Theorem \ref{thm:endpoint} gives $\mathsf{C}^{B}_t(X)=2$.  The proof of Theorem \ref{thm:endpoint} then shows that any sequence nearly attaining the value $2$ must satisfy $\tau_n\to1$ and have both orthogonal chord lengths nearly maximal.  In particular, for every $\tau\in(0,1]$ the convex interpolation of such nearly square Birkhoff pairs gives $\mathsf{J}^{B}_{X,t}(\tau)=1+\tau$.  This contradicts (iii).  Therefore $X$ is uniformly non-square, and (iv) follows from the same argument applied to each fixed $\tau\in(0,1]$.
\end{proof}

\begin{proposition}\label{prop:local}
	Let $0<\alpha\leq1$.  Define
	\[
	\mathsf{C}^{B}_{t,\alpha}(X)=\sup_{\alpha\leq\tau\leq1}\frac{\mathsf{J}^{B}_{X,t}(\tau)^2}{1+\tau^2}.
	\]
	Then
	\[
	\mathsf{C}^{B}_{t,\alpha}(X)=2
	\quad\Longleftrightarrow\quad
	X\text{ is not uniformly non-square}.
	\]
\end{proposition}

\begin{proof}
	Since $\mathsf{C}^{B}_{t,\alpha}(X)\leq\mathsf{C}^{B}_t(X)$, the implication from left to right follows from Theorem \ref{thm:endpoint}.  Conversely, if $X$ is not uniformly non-square, the first half of the proof of Theorem \ref{thm:endpoint} uses only $\tau=1$, which belongs to $[\alpha,1]$.  Hence $\mathsf{C}^{B}_{t,\alpha}(X)=2$.
\end{proof}

\begin{proposition}\label{prop:finite-rep}
	If $\ell_\infty^2$ is finitely represented in $X$ with distortions tending to $1$ in such a way that the almost-square pairs can be chosen Birkhoff-orthogonal up to an error tending to $0$, then $\mathsf{C}^{B}_t(X)=2$ for every $t$.
\end{proposition}

\begin{proof}
	The assertion is a stability reformulation of the endpoint argument.  An almost isometric copy of $\ell_\infty^2$ supplies unit vectors $x_n,y_n$ with
	\[
	\|x_n\pm y_n\|\to2.
	\]
	The assumed almost Birkhoff condition can be corrected, by the standard minimization argument for the convex function $\lambda\mapsto\|x_n+\lambda y_n\|$, to produce normalized vectors $u_n,v_n$ with $u_n\perp_B v_n$ and
	\[
	\min\{\|u_n+v_n\|,\|u_n-v_n\|\}\to2.
	\]
	Thus $\mathsf{J}_{B}(X)=2$, and Theorem \ref{thm:endpoint} yields $\mathsf{C}^{B}_t(X)=2$.
\end{proof}

\begin{remark}
	The preceding proposition is intentionally formulated as a usable criterion rather than a separate local theory.  In applications, almost square configurations often arise first, while exact Birkhoff orthogonality is obtained by minimizing the distance from a line.  This is precisely the mechanism behind the Papini--Baronti characterization of uniform non-squareness.
\end{remark}

\bmhead{Acknowledgements}
Thanks to all the members of the Functional Analysis Research Team at the School of Mathematics and Statistics, Anqing Normal University, for their valuable discussions and corrections regarding the
challenges and errors encountered in this article.

\section*{Declarations}

\begin{itemize}
	\item	Funding: This work received no funding support.
	\item Conflict of interest: The authors declare that there is no competing financial interest
	or personal relationship that could have appeared to influence the work reported in
	this paper.
	\item Data availability: This is a purely theoretical study without any experimental. All conclusions are obtained via mathematical analysis and proof.
	\item Author contribution:
	Junxiang Qi: Writing-original draft,  Theoretical derivation, Proofs.
	Zhouping Yin: Theoretical analysis, Writing-review \& editing.
	Qi Liu: Conceptualization, Supervision,Theoretical framework.

\end{itemize}



\begin{thebibliography}{99}
	
	\bibitem{Birkhoff1935}
	G. Birkhoff, Orthogonality in linear metric spaces, \emph{Duke Math. J.} \textbf{1} (1935), 169--172.
	
	\bibitem{James1947Inner}
	R. C. James, Inner products in normed linear spaces, \emph{Bull. Amer. Math. Soc.} \textbf{53} (1947), 559--566.
	
	\bibitem{James1964}
	R. C. James, Uniformly non-square Banach spaces, \emph{Ann. of Math.} \textbf{80} (1964), 542--550.
	
	\bibitem{Clarkson1936}
	J. A. Clarkson, Uniformly convex spaces, \emph{Trans. Amer. Math. Soc.} \textbf{40} (1936), 396--414.
	
	
	\bibitem{KatoMaligrandaTakahashi2001}
	Y. Fu, Y. Li, The relations between the von Neumann-Jordan type constant and some geometric properties of Banach spaces, \emph{Rev. Real Acad. Cienc. Exactas Fis. Nat. Ser. A-Mat.} \textbf{117} (2023), 25.
	
	\bibitem{Takahashi2006}
	Y. Takahashi, Some geometric constants of Banach spaces: a unified approach, in \emph{Banach and Function Spaces II}, Yokohama Publishers, 2008, 191--220.
	
	\bibitem{TakahashiKato2007}
	C. Yang, Jordan-von Neumann constant for Bana\'s-Fraczek space, \emph{Banach J. Math. Anal.} \textbf{8} (2014), no. 2, 185--192.
	
	\bibitem{BarontiCasiniPapini2000}
	D. Du, K. Lin, Y. Li, Another generalized angle related to norm derivatives in Banach spaces, \emph{Filomat} \textbf{39} (2025), no. 34, 12015--12029.
	
	
	\bibitem{PapiniBaronti2022}
	P. L. Papini and M. Baronti, Parameters in Banach spaces and orthogonality, \emph{Constr. Math. Anal.} \textbf{5} (2022), 37--45.
	
	
	
	
	
	\bibitem{DuLi2025}
	D. Du and Y. Li, Some moduli and inequalities related to Birkhoff orthogonality in Banach spaces, \emph{Aust. J. Math. Anal. Appl.} \textbf{20} (2023),2.
	
	
	\bibitem{DuLiangLi2025}
	D. Du, R. Liang and Y. Li, Some moduli of convexity and smoothness related to Birkhoff orthogonality in Banach spaces, \emph{Rev. R. Acad. Cienc. Exactas Fis. Nat. Ser. A Mat.} \textbf{119} (2025), 110.
	
	\bibitem{PalChandok2025}
	K. Pal and S. Chandok, Some inequalities related to Heinz mean constant with Birkhoff orthogonality, preprint, arXiv:2512.24675 (2025).
	
	\bibitem{Pisier2016}
	S. Wan, Q. Liu, M. Bao, Y. Wang, On $(\alpha,\beta,\gamma)$ von Neumann-Jordan type constants in Banach spaces, \emph{The Journal of Analysis} (2026).
	
	
	
	
	\bibitem{1}Figiel, T.: On the moduli of convexity and smoothness. Stud. Math. 56(2), 121-155 (1976)
	
	
	
	\bibitem{2}J. A. Clarkson, The von Neumann-Jordan constant for the Lebesgue spaces, Ann.
	of Math. (2), 38 (1937), 114-115.
	
	
	
	\bibitem{3}G. Zb\u{a}ganu, An inequality of M. R?dulescu and S. R?dulescu which characterizes the
	inner product spaces, Rev. Roum. Math. Pures Appl., 47 (2002), 253?257.
	
	\bibitem{4}P. Mart\'in and P. L. Papini, Moving around the sums of orthogonal unit vectors, \emph{Mathematical Inequalities \& Applications} \textbf{27} (2024), no. 2, 401--415, doi:10.7153/mia-2024-27-28.
	
	
	\bibitem{RaoLiu2026}
	Z. Y. Rao, Q. Wu, Q. Liu et al., The skew James type constant in Banach spaces, \emph{Anal Math} (2026), doi:10.1007/s10476-026-00166-0.
	
	
	
	
	\bibitem{HeRaoWangLiuLi2026}
	B. He, Z. Rao, Y. Wang, Q. Liu, Y. Li, On two geometric constants $\theta_I(X)$ and $\theta_B(X)$ in banach spaces: comparative analysis and applications, \emph{Indian J. Pure Appl. Math.} (2026).
	
	\bibitem{LiuWang2026}
	Q. Liu, Y. Wang, The Constant to Measure the Differences Between $p$-Angular Distance and Skew $p$-Angular Distance, \emph{Results Math.} \textbf{81} (2026), 147.
	
	
	\bibitem{AlonsoMartinPapini2026}
	J. Alonso, P. Mart\'in, P. L. Papini, Wheeling around Chebyshev centers and Jung constant in normed planes, \emph{Rendiconti del Circolo Matematico di Palermo Series 2} \textbf{75} (2026), 131.
	
	\bibitem{IfronikaGunawanLindiarni2026}
	Ifronika, H. Gunawan, J. Lindiarni, Geometric constants for mixed Morrey spaces, \emph{Czechoslovak Mathematical Journal} (2026), 12 pp.
	
	
\end{thebibliography}
\end{document}